\documentclass[12pt]{article}
\usepackage[utf8]{inputenc}

\usepackage{amsmath}
\usepackage{mathtools}
\usepackage{amssymb}
\usepackage{amsthm}
\usepackage{graphicx}
\usepackage{amscd}
\usepackage{epic, eepic}
\usepackage[lowtilde]{url}
\usepackage{color}
\usepackage[utf8]{inputenc} 
\usepackage{makebox}
\usepackage{comment}
\usepackage{enumerate}   
\usepackage{hyperref}
\usepackage{todonotes}

\theoremstyle{plain}
\newtheorem{theorem}{\bf Theorem}[section]
\newtheorem{lemma}[theorem]{\bf Lemma}
\newtheorem{claim}[theorem]{\bf Claim}
\newtheorem{cor}[theorem]{\bf Corollary}

\newtheorem{conj}[theorem]{\bf Conjecture}
\newtheorem{const}[theorem]{\bf Construction}
\newtheorem{nota}[theorem]{\bf Notation}

\newtheorem{question}[theorem]{\bf Question}

\newtheorem{result}[theorem]{\bf Result}

\newcommand{\nn}{\mathbb{N}}
\newcommand{\zz}{\mathbb{Z}}

\title{The extensible No-Three-In-Line problem}

\author{Dániel T. Nagy\thanks{Alfréd Rényi Institute of Mathematics. The author is partially supported by NKFIH grants FK 132060 and PD 137779 and by the J\'anos Bolyai Research Fellowship of the Hungarian Academy of Sciences.
		E-mail: {\tt nagydani@renyi.hu}}
	\and Zolt\'an L\'or\'ant Nagy\thanks{ELTE Linear Hypergraphs Research Group and ELTE GAC  Research Group,
		E\"otv\"os Lor\'and University, Budapest, Hungary. The author is supported by the Hungarian Research Grant (NKFI)  PD  134953 and K 124950.  	E-mail: {\tt nagyzoli@cs.elte.hu}}
	\and Russ Woodroofe\thanks{University of Primorska, Koper, Slovenia.  The author is supported in part by the Slovenian Research Agency (ARRS) research program P1-0285 and research projects J1-9108, N1-0160, J1-2451, and J3-3003.
		E-mail: {\tt russ.woodroofe@famnit.upr.si } Webpage: \url{https://osebje.famnit.upr.si/\string~russ.woodroofe/}}
}
\date{}

\begin{document}
	\maketitle

	\begin{abstract} 
		
		The classical No-Three-In-Line problem seeks the maximum number of
		points that may be selected from an $n\times n$ grid while avoiding
		a collinear triple. The maximum is well known to be linear in $n$.
		Following a question of Erde, we seek to select sets of large density
		from the infinite grid $\zz^{2}$ while avoiding a collinear triple. We show the existence of such
		a set which contains $\Theta(n/\log^{1+\varepsilon}n)$ points in $[1,n]^{2}$ for all $n$, where $\varepsilon>0$ is an arbitrarily small real number.
		We also give computational evidence suggesting that a set of lattice points may exist
		that has at least $n/2$ points on every large enough $n\times n$ grid.
		\medskip
		
	\end{abstract}

	\section{Introduction}
	A set of points in the plane are said to be in \emph{general position} if no three of the points lie on a common line.  Motivated by a  problem concerning the placement of chess pieces, Dudeney \cite{dudeney1958amusements} asked how many points may be placed in an $n \times n$ grid so that the points are in general position.  This No-Three-In-Line problem has received considerable attention: for history and background, we refer to the book of Brass, Moser, and Pach \cite{Brass2005} and that of Eppstein  \cite{Eppstein:2018}; see also \cite{por2007no} for the problem in a higher dimensional setting.
	For an upper bound, it is straightforward to see that at most $2n$ points may so be placed.  For rather small $n$, several examples have been constructed where the theoretical bound $2n$ can be attained, see e.g. \cite{anderson1979update, flammenkamp1998progress}. However, it is still an open problem to determine the answer for large $n$.
	
	Joshua Erde proposed
	the following question at the Third Southwestern German Workshop on
	Graph Theory.
	\begin{question}
		\label{que:Erde}Suppose that  $S\subseteq\zz^{2}$ is a set of grid points in
		general position. Is it true that 
		\[
		\lim\inf\frac{\left|S\cap[1,n]^{2}\right|}{n}=0?
		\]
		
	\end{question}
	
	The purpose of this paper is to give evidence suggesting that the
	answer to the question may be ``no". 
	
	While it is unknown for larger $n$ whether the upper bound $2n$ is achievable in the $n\times n$ grid,
	there are several constructions where the size of the set is a smaller multiple of
	$n$. The earliest of these is due to Erd\H{o}s (appearing in a paper
	published by Roth \cite{Roth:1951}), and uses the \emph{modular parabola},
	consisting of the points $(i,i^{2})\mod p$. If $n=p$ is a prime
	number, then this yields $n$ points in general position in $\zz^{2}\cap[1,n]^{2}$.
	If $n$ is not prime, then taking $p$ to be the largest prime before
	$n$ yields $n-o(n)$ points in general position.
	
	The best  known  general construction for the No-Three-In-Line problem is due to Hall, Jackson, Sudbery, and Wild \cite{Hall/Jackson/Sudbery/Wild:1975}.  Their construction places points on a hyperbola $xy = k \pmod p$, where $p$ is a prime slightly smaller than $n/2$, and yields $\frac{3}{2}n - o(n)$ points in general position.  
	
	Related problems in projective geometry are also of importance. For a given set of points, a line is called a secant if it intersects the point set in at least two points. Point sets having no secants of size larger than $2$ are called arcs. Following earlier work of Bose \cite{bose1947mathematical}, Segre showed that for odd $q$, every arc of $q+1$ points in Desarguesian projective geometry $PG(2,q)$
	is a conic \cite{Segre:1955} (see also \cite{Segre:1955a}). Arcs
	and higher dimensional analogues have abundant applications in cryptography
	and coding theory \cite{ball2005bounds,ball2020arcs}. There are also
	connections between the No-Three-In-Line problem and planar drawings
	of graphs \cite{Wood:2005}.
	
	Our aim in this paper is to give a bound on the growth rate of $|S\cap [1,n]^2|$. 
	Thus, in contrast with the finite grid case, we need $S$ to be dense in every square $[1,n]^2$.
	
	A trivial lower bound follows from the parabola construction $\{(x,x^2): x\in \mathbb{Z}^+\}$, giving $$|S\cap [1,n]^2|=\Omega(n^{1/2}).$$
	The constructions of Erd\H{o}s and of Hall, Jackson, Sudbery, and Wild
	give large point sets in general position in any $n\times n$ grid,
	but rely heavily on choosing a prime based on $n$. These constructions
	do not straightforwardly generalize to an infinite set, as required
	for Question~\ref{que:Erde}. Payne and Wood in \cite{Payne/Wood:2013}
	give a probabilistic construction (which can be turned into a probabilistic
	algorithm, as observed in \cite[Algorithm 9.22]{Eppstein:2018}),
	but their techniques also rely on $n$ being fixed. 
	
	Thus, Question~\ref{que:Erde} asks whether there are large sets
	of points in general position in the $n\times n$ grid, which can
	be extended nicely to larger sets of such points in larger grids (the
	\emph{Extensible No-Three-In-Line problem} of the title). This is
	connected closely to greedy approaches to the No-Three-In-Line problem,
	in which one seeks to build a large set of points in general position
	by adding points one-by-one according to some simple rubric. Although
	greedy approaches often appear to work reasonably well at building
	large point sets in computer experiments, they do not seem to be easy
	to analyze. 
	
	A subset $S$ in general position of points from some universe $U$
	is said to be \emph{saturated} if the addition of any further point from
	$U$ destroys the general position property. As greedy algorithms
	often add points until no more points can be added, they may result
	in saturated subsets in general position. Saturated subsets were considered
	by Adena, Holton, and Kelly in \cite{Adena/Holton/Kelly:1974}, and
	Martin Gardner asked a related question in his ``Mathematical Games''
	column \cite{Gardner:1976}. It is straightforward to see that any
	saturated subset has at least $\Omega(\sqrt{n})$ points. 
	
	
	Recently, Aichholzer, Eppstein, and Hainzl found a significantly improved
	lower bound on saturated subsets.
	
	\begin{theorem}[Aichholzer, Eppstein, and Hainzl \cite{Aichholzer/Eppstein/Hainzl:2023}]
		\label{thm:SaturatedLB} If $S$ is a saturated subset of points
		in general position from an $n\times n$ grid, then $\left|S\right|=\Omega(n^{2/3})$.
	\end{theorem}
	
	It follows straightforwardly that there is an infinite subset $S\subseteq\zz^{2}$
	in general position so that $\left|S\cap[1,n]^{2}\right|=\Omega(n^{2/3})$.
	It is worth remarking that the result of Theorem~\ref{thm:SaturatedLB}
	is markedly different from the situation in a projective plane: for
	any sufficiently large projective plane having $q^{2}+q+1$ points,
	Kim and Vu \cite{kim2003small} showed that there is a saturated  (or maximal) arc
	of order $\sqrt{q}\cdot\log^{c}q$ (for some fixed $c$). This indicates that completing a given structure to gain a dense set in a larger square is difficult in general.
	

	
	We show that the asymptotic growth of the size of the point set for the problem on the infinite grid can, in fact, be almost linear.
	
	\begin{theorem} \label{thm:main}
		For any $\varepsilon>0$, it is possible to construct
		a set $S\subseteq\zz^{2}$ of grid points in general position with
		$$\left|S\cap[1,n]^{2}\right|=\Theta(n/\log^{1+\varepsilon}n).$$  In particular,
		it holds that 
		\[
		\lim\inf\frac{\left|S\cap[1,n]^{2}\right|}{n/\log^{1+\varepsilon}n}>0.
		\]
	\end{theorem}
	
	The main ingredients of the construction underlying Theorem~\ref{thm:main} are as follows.  We place separated copies of the parabola construction of Erdős along the curve $x/\log^\varepsilon x$.  Specifically, for each value $x=2^n$, we consider a square placed near the point $(2^n, n^\varepsilon)$ having side length a small multiple of $2^n/n^{1+\varepsilon}$.  We choose a suitable prime $p_n$ for each integer $n$, and place a translated copy of Erdős's parabola construction with respect to $p_n$ in the square. Finally, we delete those points from each such parabola that would form a collinear triple with points to their left.
	
	By concavity, any line intersects the curve $x/\log^\varepsilon$ in at most two points.  If we choose the squares in the construction to be small enough, then (as we will see) a line intersects at most two of the squares.
	Thus to avoid collinear triples, it is enough to delete a point from each line which intersects the previously defined point set in one point and the parabola in the $n$th square in two points, or vice versa. 
	By bounding from above the  number of deleted points, we obtain $\Theta(\frac{N}{\log^{1+\varepsilon} N})$ lattice points in general position for each $[1,N]^2$, verifying Theorem 1.3.
	
	\paragraph{}
	In addition to the construction of Theorem~\ref{thm:main}, we also 
	have some computational results. The \emph{lexicographic greedy
		construction} on a subset $U$ of $\nn^{2}$ proceeds from left to
	right across $U$. On each vertical line, we add to $S$ the lowest
	point that is not on the same line with any two already-placed points
	(if such a point exists). The lexicographic greedy construction on
	the (infinite height) grid has been previously examined in OEIS sequences
	A236335 and A236266 \cite{OEIS}. We consider the lexicographic greedy
	construction over the triangular region consisting of the points with
	$y\leq x$ in the first quadrant of $\zz^{2}$. In computer experiments, this
	appears to stably yield slightly more than $0.8n$ points in general
	position in $[1,n]^{2}$. We also consider the variant where we only
	take points on vertical lines having even $x$-intercept. In computer experiments,
	this even variant appears to find points without fail, yielding exactly
	$n$ points in general position in $[1,2n]^{2}.$ 
	
	The organisation of the paper is as follows. We finish this section by presenting the outline of Theorem~\ref{thm:main}. In Section 2 we introduce the notation we wish to use throughout the paper and recall some background results. Section 3 is devoted to the proof of Theorem~\ref{thm:main}. Finally, in Section 4 we describe computer experiments with the lexicographic greedy construction, which suggest that the answer for Question \ref{que:Erde} is ``no".  Motivated by these experiments, we formulate novel conjectures.
	
	\section{Preliminaries}
	
	\begin{nota} When speaking about the lattice points of a grid, 
		$[1,N]$ denotes the integers in the closed interval. 
	\end{nota}
	
	We will use (without further reference) the following well-known results from analysis and number theory.
	
	\begin{result}\label{result:harmo} The first $N$ terms of the harmonic series sum up as 
		$\sum_{i=1}^N \frac{1}{k}= \ln N + O(1).$
	\end{result}
	
	\begin{result}[Baker, Harman, Pintz \cite{Baker/Harman/Pintz:2001}] There is a prime number in the interval $[x-x^{21/40},x]$ for all large enough $x$.
	\end{result}
	
	Let us recall the details of the parabola construction of Erdős discussed in the introduction.  This construction provides $n-o(n)$ points in general position in the grid $[1,n]\times [1,n]$. We introduce a slightly generalized form.
	
	\begin{const}[The general modular parabola construction] \label{const:erdos} Let 
		$\mathcal{P}$ denote a parabola $y=(x-a)^2+b \pmod {p}$ for some prime $p$ and integers $a, b$. The graph of the parabola in $[1,p]\times [1, p]$ lies in general position, similarly to the Erdős construction. Indeed, the two extra parameters $a,b$ merely translate the graph of $y=x^2 \pmod {p}$.
	\end{const}
	
	We make two observations regarding Construction~\ref{const:erdos}. The first states that two points determine a unique parabola having leading coefficient $1$. The second notices that that there is at most one point pair of the parabola with a given direction and Euclidean distance.
	
	\begin{claim}\label{claim:nagyon_para}
		\begin{enumerate}
			\item[]
			\item  Let $\mathbb{F}$ be a field, and $(x_0, y_0), (x_1, y_1)\in \mathbb{F}\times \mathbb{F}.$ Then the system of equations \[
			\left\{
			\begin{array}{ll}
				(x_0-a)^2+b=y_0\\
				(x_1-a)^2+b=y_1
			\end{array}
			\right.
			\] has a unique solution $(a,b)$, provided that $x_0\neq x_1$.
			
			\item  Let $\mathcal{P}$ be the parabola $y= (x-a)^2+b \pmod p$ on the grid $[0,p-1]\times [0,p-1] \subseteq \mathbb{Z}^2$.  Choose  any ordered  pair of  positive numbers $s,t$. There is at most one pair of points $Q_1, Q_2$ lying on $\mathcal{P}$ for which the line segment $Q_1Q_2$  has slope $s$ and  length $t$.
		\end{enumerate}
	\end{claim}
	
	\begin{proof}
		To prove the first part, subtract the second equation from the first to get a linear equation in $a$, which uniquely determines $a$.  The uniqueness of $b$ now follows immediately.
		
		To prove the second part, suppose that we have four points on the parabola,  $Q_1, Q_2, R_1, R_2,$ such that $R_1R_2$ is a translate of $Q_1Q_2$.  That is, the points have the form $Q_1=(x_0, y_0), Q_2=(x_1, y_1)$ and $R_1=(x_0+u, y_0+v),  R_2=(x_1+u, y_1+v)$. Then the following equations hold: \[
		\begin{array}{ll}
			(x_0-a)^2+b=y_0\\
			(x_1-a)^2+b=y_1\\
			(x_0+u-a)^2+b=y_0+v\\
			(x_1+u-a)^2+b=y_1+v.\\
		\end{array}.
		\]
		However, the first two equations determine $(a,b)$ as in the proof of the first part, and the last two equations determine the same pair $(a,b)$ if and only if $u=v=0$.
	\end{proof}
	
	We will also frequently use the following bound throughout the proof of the main theorem.
	
	\begin{claim}\label{claim:sum} Let $\varepsilon\in [0,1]$ be a non-negative real number.  If $n>10$, then
		$$\sum_{m=1}^{n-1} \frac{2^m}{m^{1+\varepsilon}} \leq 2\cdot \frac{2^n}{n^{1+\varepsilon}}.$$
	\end{claim}
	
	\begin{proof}It is easy to see that the claim holds for $n=11$ for every $\varepsilon$, since it in turn follows from the  case $\varepsilon=1$ by monotonicity. To prove for every $n>11$, one can easily apply induction. Then it is enough to verify $\frac{2+1}{2\cdot 2}\leq (\frac{n}{n+1})^{1+\varepsilon}$ to prove the inductional hypothesis.
	\end{proof}
	
	\section{Proof of the main result}\label{sec:mainproof}
	
	\subsection{Construction}
	We begin by describing our construction.
	
	The building blocks of the construction will be copies of the parabola of Construction~\ref{const:erdos}, taken over various primes $p$.  We will place these copies along the curve $x / \log^\varepsilon x$ (for some fixed small $\varepsilon$).  By concavity, no line passes through more than 2 points of this curve.  If we choose our parabolas to be small enough, then no line will pass through more than two of the copies, and we can then control the lines passing through two copies.
	
	The details are as follows.  It may be helpful to refer to Figure~\ref{fig:consecutive}.
	
	\begin{const}[Main construction]\label{const:box}
		Fix a real number $0 < \varepsilon < 1$.  Fix also a real number $c$, which will be at least $12/\varepsilon$ (see Lemma~\ref{lem:consecutive2} below).
		\begin{enumerate}  
			\item  For each sufficiently large $n$, consider an axis-parallel square
			$Q_{n}$ with sides of length $\left\lfloor \frac{2^{n}}{cn^{1+\varepsilon}}\right\rfloor $,
			and top left corner at $(2^{n},\left\lfloor \frac{2^{n}}{n^{\varepsilon}}\right\rfloor )$.
			Let $p_{n}$ be the largest prime smaller than $\left\lfloor \frac{2^{n}}{cn^{1+\varepsilon}}\right\rfloor $. 
			
			\item  In each $Q_{n}$, consider a translated copy of the parabola $(x-a_{n})^{2}+b_{n}\mod p_{n}$
			(for some choice of parameters $a_{n}$ and $b_{n}$), extending down
			and right from the top left corner. Denote this parabola as $\mathcal{P}_{n}$.
			
			\item Our point set $S$ will consist of a subset of $\bigcup_{n}\mathcal{P}_{n}$.
		\end{enumerate}
		
	\end{const}
	
	We will select the subset $S$ of $\bigcup\mathcal{P}_{n}$ iteratively.
	Assuming that we have already selected points in $\mathcal{P}_{1},\dots,\mathcal{P}_{n-1}$,
	we will select the parameters $a_{n},b_{n}$ to minimize the number
	of collinear triples spanned by points of $\mathcal{P}_{n}$ and $S\cap\left(\bigcup_{i=1}^{n-1}Q_{i}\right)$.
	We remove one point of each such collinear triple from $\mathcal{P}_{n}$,
	and add the remaining points to $S$.
	
	The floors in the construction make no difference in the asymptotic situation, and we will generally omit them to make our calculations easier to follow.

	\begin{figure}[ht]
		\centering
		\includegraphics[width=1.0\textwidth]{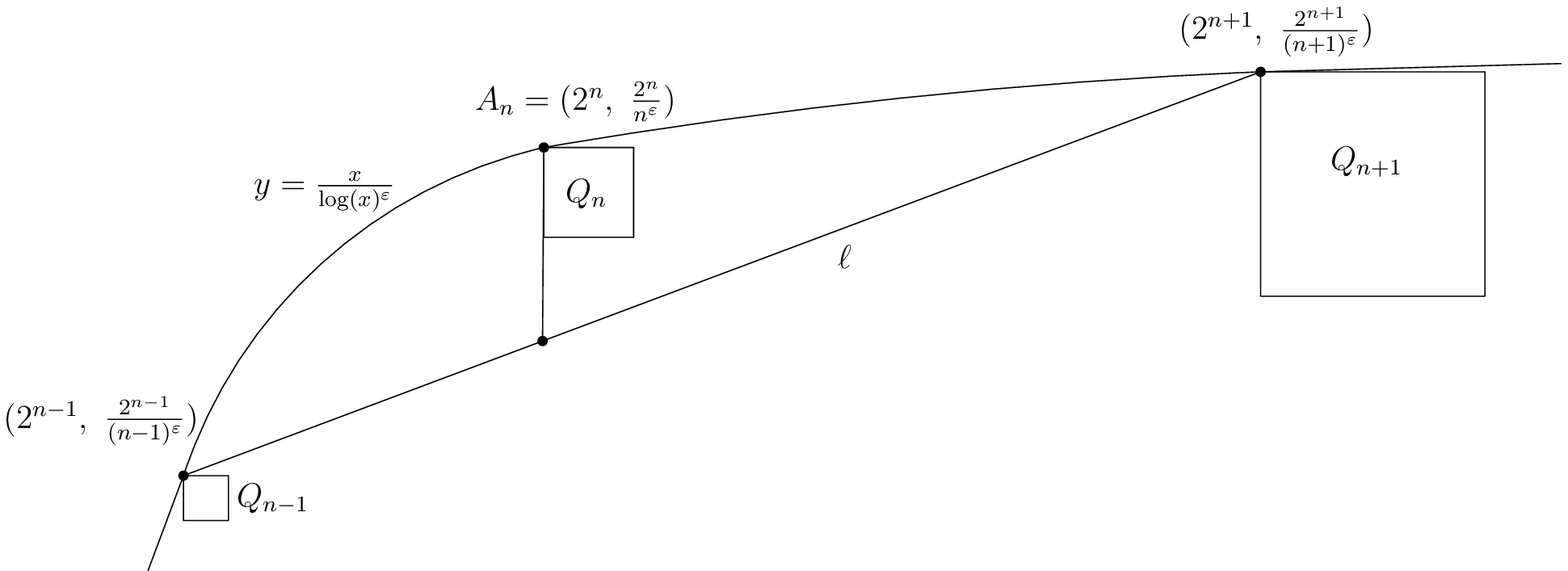}
		\caption{Illustration for Lemma~\ref{lem:consecutive2}}
		\label{fig:consecutive}
	\end{figure}
	
	\subsection{Bounding lemmas}
	Fix notation as in Construction~\ref{const:box}. We first show
	that no line intersects more than two copies of the modular parabola
	gadget.
	
	\begin{lemma}\label{lem:consecutive2}
		If $c$ is at least $c_{0}(\varepsilon)$, then no line intersects more
		than two of the squares $Q_{n}$.
	\end{lemma}
	
	\begin{proof}
		Suppose that $m<n<k$, and that $\ell$ is a line passing through
		the squares $Q_{m}$ and $Q_{k}$. We desire to show that $Q_{n}$ lies
		strictly above $\ell$. 
		
		Since the $x$ coordinates for $Q_{n}$ lie strictly between those
		for $Q_{m}$ and $Q_{k}$, and since the top edge of $Q_{m}$ lies
		strictly below that of $Q_{k}$, we can assume without loss of generality
		that $\ell$ intersects $Q_{m}$ and $Q_{k}$ at their top left corners.
		By concavity of the curve $x/\log^{\varepsilon}x$, we may assume that
		$m=n-1$ and $k=n+1$.
		
		We now want to show that the lower right hand corner of $Q_{n}$ is above
		$\ell$. Since the slope of $\ell$ is clearly smaller than $1$,
		it suffices to look at the left side of $Q_{n}$, and show that the
		difference between $\left\lfloor \frac{2^{n}}{n^{\varepsilon}}\right\rfloor $
		and the height here of $\ell$ is at least $2\cdot\left\lfloor \frac{2^{n}}{cn^{1+\varepsilon}}\right\rfloor $ 
		(see Figure~\ref{fig:consecutive}).  
		Since $2^{n}/2^{n-1}$ and $2^{n+1}/2^{n}$ are in a $2:1$ ratio,
		this difference is (ignoring floors) 
		\[
		\frac{2^{n}}{n^{\varepsilon}}-\frac{2}{3}\cdot\frac{2^{n-1}}{(n-1)^{\varepsilon}}-\frac{1}{3}\cdot\frac{2^{n+1}}{(n+1)^{\varepsilon}}=\frac{2^{n}}{3n^{\varepsilon}}\cdot\left(3-\frac{n^{\varepsilon}}{(n-1)^{\varepsilon}}-\frac{2n^{\varepsilon}}{(n+1)^{\varepsilon}}\right),
		\]
		and so it is enough to choose $c$ so that
		\[
		3-\frac{n^{\varepsilon}}{(n-1)^{\varepsilon}}-\frac{2n^{\varepsilon}}{(n+1)^{\varepsilon}}=2\cdot\frac{(n+1)^{\varepsilon}-n^{\varepsilon}}{(n+1)^{\varepsilon}}-\frac{n^{\varepsilon}-(n-1)^{\varepsilon}}{(n-1)^{\varepsilon}}
		\]
		is at least $\frac{6}{cn}$. We recall a slight extension
		of Bernoulli's inequality, that $1+\varepsilon x>(1+x)^{\varepsilon}>1+\varepsilon x+\varepsilon(\varepsilon-1)x^{2}/2$
		when $0<\varepsilon<1$ and $-1<x<1$. Here, the second inequality follows
		from alternating series estimation when $x>0$, or from Taylor's inequality
		when $x<0$. Thus, our chain of inequalities continues: 
		\begin{align*}
			2\cdot\frac{(n+1)^{\varepsilon}-n^{\varepsilon}}{(n+1)^{\varepsilon}}-\frac{n^{\varepsilon}-(n-1)^{\varepsilon}}{(n-1)^{\varepsilon}} & >  2\cdot\frac{\varepsilon n^{\varepsilon-1}+\varepsilon(\varepsilon-1)n^{\varepsilon-2}/2}{n^{\varepsilon}+\varepsilon n^{\varepsilon-1}}-\frac{\varepsilon n^{\varepsilon-1}+\varepsilon(\varepsilon-1)n^{\varepsilon-2}/2}{n^{\varepsilon}-\varepsilon n^{\varepsilon-1}+\varepsilon(\varepsilon-1)n^{\varepsilon-2}}\\
			& = \frac{2\varepsilon}{n}\cdot\frac{1+\varepsilon n^{-1}-n^{-1}/2}{1+\varepsilon n^{-1}}-\frac{\varepsilon}{n}\cdot\frac{1+(\varepsilon-1)n^{-1}/2}{1-\varepsilon n^{-1}+\varepsilon(\varepsilon-1)n^{-2}}\\
			& = \frac{\varepsilon}{n}\cdot\left(2-\frac{1}{n+\varepsilon}-\frac{1+(\varepsilon-1)n^{-1}/2}{1-\varepsilon n^{-1}+\varepsilon(\varepsilon-1)n^{-2}}\right).
		\end{align*}
		As the last entry is greater than $\varepsilon/2n$ past very small $n$,
		we may take $c_{0}$ to be $12/\varepsilon$.
	\end{proof}
	
	\begin{lemma}\label{lem:meredek} If $\ell$ is a line which intersects $Q_m$ and $Q_n$ for $m<n$ integers, then the slope $\mu(\ell)$ satisfies
		
		$$\frac{2^{n-m}(\frac{1}{n^\varepsilon}-\frac{1}{cn^{1+\varepsilon}})-\frac{1}{m^\varepsilon}}{2^{n-m}(1+\frac{1}{cn^{1+\varepsilon}})-1}\leq \mu(\ell)\leq \frac{2^{n-m}/n^\varepsilon-\frac{1}{m^\varepsilon}+\frac{1}{cm^{1+\varepsilon}}}{2^{n-m}-1-\frac{1}{cm^{1+\varepsilon}}}.$$
		
	\end{lemma}
	\begin{proof}
		Let $A_i$ and $D_i$ denote the top left and the bottom right corner of the square $Q_i$, respectively, for $i\in\{m,n\}$. It is easy to see that the smallest possible slope corresponds to the line $A_mD_n$ and the largest possible slope corresponds to $D_mA_n$. Taking into consideration that $A_i=(2^i, \frac{2^i}{i^\varepsilon})$, and $D_i=(2^i+\frac{2^i}{ci^{1+\varepsilon}}, \frac{2^i}{i^\varepsilon}-\frac{2^i}{ci^{1+\varepsilon}})$, we get the desired bounds.
	\end{proof}
	
	We will need the following technical result on the difference between the upper and lower bounds of Lemma~\ref{lem:meredek}.
	\begin{lemma}[Technical Lemma]
		\label{lem:TechBound}Let $m<n$ be integers, and $0<\varepsilon<1$.
		If $n$ is sufficiently large, then 
		
		\[
		\frac{\frac{1}{n^{\varepsilon}}-\frac{1}{c2^{n-m}m^{\varepsilon}}+\frac{1}{c2^{n-m}m^{1+\varepsilon}}}{1-2^{n-m}-\frac{1}{c2^{n-m}m^{1+\varepsilon}}}-\frac{\frac{1}{n^{\varepsilon}}-\frac{1}{c2^{n-m}m^{\varepsilon}}-\frac{1}{cn^{1+\varepsilon}}}{1-2^{n-m}+\frac{1}{cn^{1+\varepsilon}}}<\frac{11}{cn^{1+\varepsilon}}.
		\]
	\end{lemma}
	
	\begin{cor}\label{cor:meredek}   Let $\ell$ be a line having slope $\mu(\ell)=r/q$, and that intersects $Q_m$ and $Q_n$ for $m<n$ integers.  If $n$ is sufficiently large, then $r$ can admit at most $q\cdot \frac{11}{cn^{1+\varepsilon}}$ different values.
	\end{cor}
	
	\begin{proof}[Proof (of Lemma~\ref{lem:TechBound})]
		The key step is to notice that as $2^{n}/n^{1+\varepsilon}>2^{m}/m^{1+\varepsilon}$
		for $n$ sufficiently large, also $2^{n-m}m^{1+\varepsilon}>n^{1+\varepsilon}$.
		Write $A$ for $\frac{1}{n^{\varepsilon}}-\frac{1}{c2^{n-m}m^{\varepsilon}}$
		and $B$ for $1-2^{n-m}$. Then the left-hand side of the inequality is 
		\[
		\frac{A\cdot\left(\frac{1}{cn^{1+\varepsilon}}+\frac{1}{c2^{n-m}m^{1+\varepsilon}}\right)+\frac{1}{c2^{n-m}m^{1+\varepsilon}}\cdot\left(B+\frac{1}{cn^{1+\varepsilon}}\right)+\frac{1}{cn^{1+\varepsilon}}\cdot\left(B-\frac{1}{c2^{n-m}m^{1+\varepsilon}}\right)}{B^{2}(1\pm o(1))}.
		\]
		Notice that $A\to0$, as do the smaller terms on the bottom, and also
		that $\frac{1}{2}<B<1$. Applying the key step, we thus bound the
		left-hand side for large enough $n$ by 
		\[
		\frac{\frac{1}{10}\cdot\frac{2}{cn^{1+\varepsilon}}+\frac{1}{cn^{1+\varepsilon}}\cdot\frac{11}{10}+\frac{1}{cn^{1+\varepsilon}}\cdot1}{1/5}<\frac{11}{cn^{1+\varepsilon}}.\qedhere
		\]
	\end{proof}
	
	\subsection{Red points and blue pairs}\label{subsec:RedBlue}
	
	Lemma~\ref{lem:consecutive2} says that any line passing through three
	points in our construction must pass through two points in one square,
	and at least one point in another. Lemma~\ref{lem:meredek} and  
	Corollary~\ref{cor:meredek} will let us say that this is rare.
	
	\begin{lemma}\label{lem:linesum}
		Let $m<n$, where $n$ is sufficiently large.  The number of collinear triples determined by two points from $S \cap Q_m$ and one lattice point of $Q_n$ is at most $$\frac{8 \cdot2^{m+n}}{c^3n^{2+2\varepsilon}m^{1+\varepsilon}}\cdot n.$$
	\end{lemma}
	
	\begin{proof} 
		We first count the number of lines $\ell$ intersecting $S \cap Q_m$ in two points according to the denominator of their slope.  Let $\ell$ have slope $\mu(\ell)=\frac{r}{q}$, where $r$ and $q$ are relatively prime.  Consider the projection onto the $x$ axis of the line segment between the two points of intersection.  The length $d$ of this projection must be a multiple of $q$, at most $\frac{2^m}{ c\cdot m^{1+\varepsilon}}$.  Moreover, it follows from Claim~\ref{claim:nagyon_para} that $d$ and $\frac{r}{q}$ uniquely determine the two points on $S \cap Q_m$, so each value of $d$ occurs at most once.
		
		Thus, we have at most $\frac{2^m}{q\cdot c\cdot m^{1+\varepsilon}}$ pairs from $S \cap Q_m$ determining a line of slope $r/q$.  If $\ell$ also passes through $Q_n$, then by Corollary~\ref{cor:meredek} there are at most $q\cdot\frac{11}{cn^{1+\varepsilon}}$ possible values of $r$.  By examining the $x$ coordinate, there are at most $\frac{2^n}{q\cdot c\cdot n^{1+\varepsilon}}$ lattice points in $Q_n$ lying on $\ell$.
		
		We finish the proof by summing over possible values of $q$, noticing that $11\cdot\ln2<8$ and taking into consideration Result~\ref{result:harmo}:
		\begin{align*}
			\sum_{q<\frac{2^{m}}{cm^{1+\varepsilon}}}\frac{1}{q}\cdot\frac{2^{n}}{cn^{1+\varepsilon}}\cdot\frac{2^{m}}{cm^{1+\varepsilon}}\cdot\frac{11}{cn^{1+\varepsilon}} & \leq\frac{11\cdot2^{m+n}}{c^{3}n^{2+2\varepsilon}m^{1+\varepsilon}}\cdot(\ln(\frac{2^{m}}{cm^{1+\varepsilon}})+O(1))\\
			& <\frac{8\cdot2^{m+n}}{c^{3}n^{2+2\varepsilon}m^{1+\varepsilon}}\cdot n.\qedhere
		\end{align*}
	\end{proof}
	
	We color as {\em red} the points in $Q_n$ that are on the same line with a pair of points in $S \cap Q_m$ for some $m<n$.  By summing the bound of Lemma~\ref{lem:linesum} over all $m<n$ and applying Claim~\ref{claim:sum}, we see that:
	\begin{cor}\label{cor:linesum}
		The number of red points in $Q_n$ is at most
		\begin{equation}\label{eq:linesum1_sum} \sum_{m<n}  \frac{8 \cdot2^{m+n}}{c^3n^{2+2\varepsilon}m^{1+\varepsilon}}\cdot n < \frac{16 \cdot2^{n+n}}{c^3n^{2+2\varepsilon}n^{1+\varepsilon}}\cdot n.
		\end{equation}
	\end{cor}
	We notice that the ratio between the number of red points and the number of all points of $Q_n$  is at most $\frac{16}{cn^\varepsilon}$.  Thus, it is clear that we will be able to choose $a_n, b_n$ as in Construction~\ref{const:box} to avoid all but an insignificant number of red points.
	
	However, we must now also consider lines that intersect two points in $Q_n$, together with some point in $S\cap Q_m$ for some $m<n$.  In analogy with our red points, we call a pair of lattice points in $Q_n$ \emph{blue} if at least one of the pair is not red, and if the two are on a common line with a point of $S \cap Q_m$ for some $m<n$.  Thus, to prove Theorem~\ref{thm:main}, it will suffice to show that there is a choice of $a_n, b_n$ that simultaneously avoids all but a small number of both red points and blue pairs.
	
	\begin{lemma}\label{lem:linesum2}
		The number of blue pairs in $Q_n$ is at most $\ \displaystyle{\frac{8\cdot 2^{3n}}{c^4 n^{3 + 4\varepsilon }}}.$ 
	\end{lemma}
	
	\begin{proof}
		Consider a fixed $m<n$, and fix the denominator $q$ of the slope $r/q$ of a line $\ell$ that intersects $S \cap Q_m$ and at least one blue pair of $Q_n$.  As in the proof of Lemma~\ref{lem:linesum}, Corollary~\ref{cor:meredek} gives that there are at most $q\cdot \frac{11}{cn^{1+\varepsilon}}$ choices for $r$.
		
		If $\ell$ would pass through more than one point of $Q_m$, then all points on $\ell$ in $Q_n$ would be colored red, and in particular none of them would be in a blue pair.  Thus, the line $\ell$ is determined (given $r$) by its intersection with $S \cap Q_m$, and we have at most $|S \cap Q_m| \leq \frac{2^m}{cm^{1+\varepsilon}}$ such lines.  Finally, on each line we have at most 
		$${\binom{\frac{2^{n}}{q\cdot cn^{1+\varepsilon}}}{2}}<\frac{2^{2n}}{2q^{2}c^{2}n^{2+2\varepsilon}}$$
		points from $Q_n$.  Combining these choices, we have at most 
		$$\frac{1}{q}\cdot\frac{11}{cn^{1+\varepsilon}}\cdot\frac{2^{m}}{cm^{1+\varepsilon}}\cdot\frac{2^{2n}}{2c^{2}n^{2+2\varepsilon}}$$
		pairs relative to a fixed $q$ and $m$.  We now notice that as $\ell$ passes through two points of $Q_n$, we must have $q < \frac{2^n}{cn^{1+\varepsilon}}$.  Summing over $q$ and $m$, in view of Result~\ref{result:harmo} and Claim~\ref{claim:sum}, we obtain
		
		\[
		\sum_{m=1}^{n-1}\sum_{q<\frac{2^{n}}{cn^{1+\varepsilon}}}\frac{1}{q}\cdot\frac{2^{m}}{cm^{1+\varepsilon}}\cdot\frac{11\cdot2^{2n}}{2c^{3}n^{3+3\varepsilon}} \leq\sum_{m}\frac{2^{m}}{cm^{1+\varepsilon}}\cdot\frac{8\cdot2^{2n}}{2c^{3}n^{3+3\varepsilon}}\cdot n \leq\frac{8\cdot2^{3n}}{c^{4}n^{3+4\varepsilon}}.\qedhere
		\]
	\end{proof}
	
	\subsection{Proof of Theorem \ref{thm:main}}
	We build our point set as in Construction~\ref{const:box}, using a value of $c$ large enough to satisfy the condition of Lemma~\ref{lem:consecutive2}.  It follows from Claim~\ref{claim:sum} that $\sum p_n$ has the desired growth rate.  Thus, provided the number of deleted points from each $Q_n$ is small, our iteratively-defined set $S$ will satisfy the requirement of the theorem.  We will show that for every $n\in \mathbb{N}$, there is a choice of parameters $a_n, b_n$, so that the number of points deleted from each $Q_n$ is of the order $p_n/n^{\varepsilon}$.
	
	Lemma~\ref{lem:consecutive2} implies that, due to the sizes and positions of the squares $Q_n$, there will be no collinear triples formed by points of three distinct squares $Q_i$.  Certain points (the red points of Section~\ref{subsec:RedBlue}) or pairs of points (the blue pairs of the same section) from $Q_n$ may form a collinear triple with a point or points of $S\cap Q_m$ for some $m<n$.  
	
	We will show that the number of such collinear triples is small.  We will use a double
	counting argument on the set $\mathcal{I}$ of pairs $(\mathcal{P},T)$,
	where $\mathcal{P}$ is a parabola modulo $p_{n}$ in $Q_{n}$, and
	$T$ is a collinear triple on $\mathcal{P}\cup(S\cap Q_{m})$ for
	some $m<n$.
	
	We first count $\mathcal{I}$ by collinear triples. Each red point
	of $Q_{n}$ determines a collinear triple with two points from $S\cap Q_{m}$
	(some $m<n)$, and lies on exactly $p_{n}$ parabolas. Each blue pair
	of $Q_{n}$ determines a collinear triple with some point from $S\cap Q_{m}$
	(some $m<n$), and lies on a unique parabola. It follows that 
	\begin{align*}
		\left|\mathcal{I}\right| & =\left|\{(\mathcal{P},C):C\in\mathcal{P}\mbox{ is red}\}\right|+\left|\{(\mathcal{P},\{C,D\}):\{C,D\}\subseteq\mathcal{P}\text{ is a blue pair}\}\right|\\
		& \leq\frac{16\cdot2^{2n}}{c^{3}n^{2+3\varepsilon}}\cdot p_{n}+\frac{8\cdot2^{3n}}{c^{4}n^{3+4\varepsilon}}.
	\end{align*}
	
	We now count $\mathcal{I}$ by parabolas. Indeed, as there are $p_{n}$
	choices for each of $a$ and $b$, the mean number of pairs $(\mathcal{P}(a,b),T)$
	in $\mathcal{I}$ for a choice of parameters $(a,b)$ is $\left|\mathcal{I}\right|/p_{n}^{2}$.
	In particular, there is some choice $a_{n},b_{n}$ of parameters for
	a parabola in $Q_{n}$ with at most 
	\[
	\left(\frac{16\cdot2^{2n}}{c^{3}n^{2+3\varepsilon}}\cdot p_{n}+\frac{8\cdot2^{3n}}{c^{4}n^{3+4\varepsilon}}\right)\bigg/p_{n}^{2}=O\left(\frac{2^{n}}{n^{1+2\varepsilon}}\right)
	\]
	pairs in $\mathcal{I}$ having $\mathcal{P}_{n}:=\mathcal{P}(a_{n},b_{n})$
	in the parabola entry. 
	
	We now form $S\cap Q_{n}$ by selecting the parabola $\mathcal{P}_{n}$,
	and deleting a point from each collinear triple $T$ so that $(\mathcal{P}_{n},T)$
	is in $\mathcal{I}$. Thus, we delete from $\mathcal{P}_{n}$ all
	incident red points, and one point from each incident blue pair. As
	$p_{n}$ is only slightly smaller than $\frac{2^{n}}{cn^{1+\varepsilon}}$,
	the number of points in the constructed $S\cap Q_{n}$ is $\frac{2^{n}}{cn^{1+\varepsilon}}\cdot(1-o(n))$.
	This suffices to complete the proof.
	
	\section{Greedy lexicographic construction}
	\begin{figure}[h]
		\centering
		\includegraphics[width=0.7\textwidth,trim={0 0 0 1.3cm},clip]{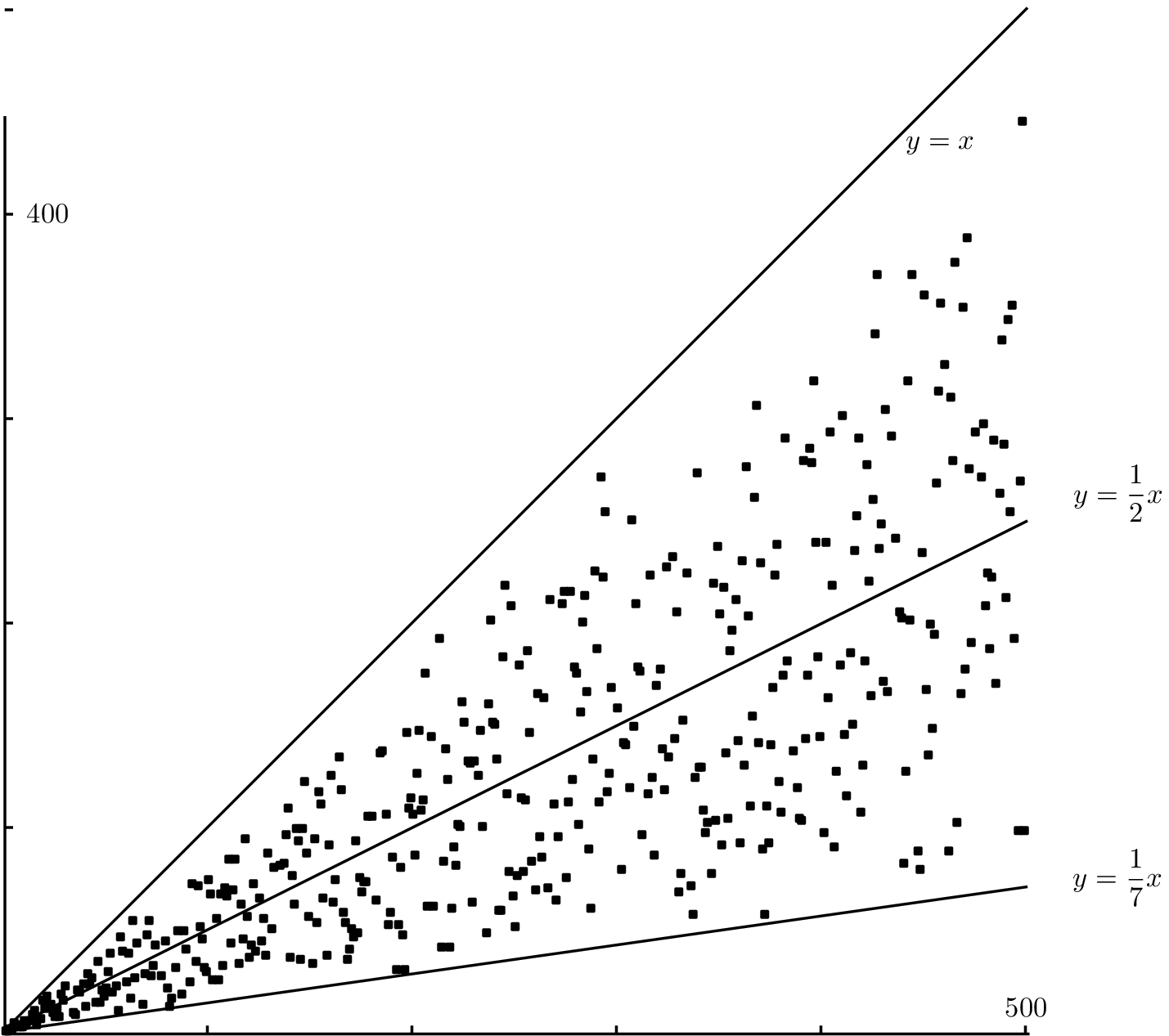}
		\caption{Construction $S_{\mathrm{lex}<}$ for $n=500$}\label{fig:Slex}
	\end{figure}
	
	A concrete and deterministic construction for large grid sets that are in general position may be described as follows.  We iteratively look at each vertical line $x=1$, $x=2$, etc.  At each line $x=i$, we record the least positive $j$ so that the point $(i,j)$ is not on a common line with any two points to its left.  The limit of this process yields an infinite set $S_{\mathrm{lex}}$ of triple-wise non-collinear integer points.  This construction has been previously examined in OEIS sequence A236335, and similar constructions are the subject of sequences A236266, A179040 \cite{sloane2007line}.
	
	Since we are interested in high density in squares $[1, n]^2$, we vary the lexicographic construction slightly to require $j < i$ at each step, yielding a set $S_{\mathrm{lex}<}$.  Experimental evidence suggests that $[1, n]^2 \cap S_{\mathrm{lex}<}$ is of approximate size that is slightly larger than $0.8 \cdot n$.  Refer to Table~\ref{tab:table_1} for densities of $S_{\mathrm{lex}<}$ at several values of $n$, or to Figure~\ref{fig:Slex} for a plot of the first points in $S_{\mathrm{lex}<}$.
	
	\begin{table}[h!]
		\centering
		\begin{tabular}{|r||l|l|l|l|l|l|l|l|l|l|l|}
			\hline 
			$n$ & 100 & 200 & 300 & 400 & 500 & 1000 & 2000 & 3000 & 4000 & 5000 & 10000\tabularnewline
			\hline 
			\# points $S_{\mathrm{lex}<}$ & $81$ & $166$ & $254$ & $340$ & $424$ & $830$ & $1678$ & $2515$ & $3353$ & $4197$ & $8385$\tabularnewline
			\hline 
			density \% $n$ & $81\%$ & $83$ & $84.6$ & $85$ & $84.8$ & $83$ & $83.9$ & $83.8$ & $83.8$ & $83.9$ & $83.9$\tabularnewline
			\hline 
			\# points $S_{\mathrm{mod2lex}}$ & $50$ & $100$ & $150$ & $200$ & $250$ & $500$ & $1000$ & $1500$ & $2000$ & $2500$ & $5000$\tabularnewline
			\hline 
			density \% $n$ & $50\%$ & $50$ & $50$ & $50$ & $50$ & $50$ & $50$ & $50$ & $50$ & $50$ & $50$\tabularnewline
			\hline 
		\end{tabular}
		\caption{Density achieved by the lexicographic greedy construction}\label{tab:table_1}
	\end{table}
	
	A variant of some interest is to look only at the even lines $x=2$, $x=4$, etc.,  and to apply the lexicographic construction on these, again with the requirement that $j < i$ at each step. We call the obtained set $S_{\mathrm{mod2lex}}$.  Experimental evidence suggests that the lexicographic construction always yields a general position point above each even $i$; moreover, in our experiments, this point always lies so that $j < \frac{2}{3}i$ (see Figure~\ref{fig:Smod2lex}).   Rigorous analysis of this type of greedy approach appears to be difficult.
	
	\begin{figure}[t]
		\centering
		\includegraphics[width=.95\textwidth,trim={0 0 0 1cm},clip]{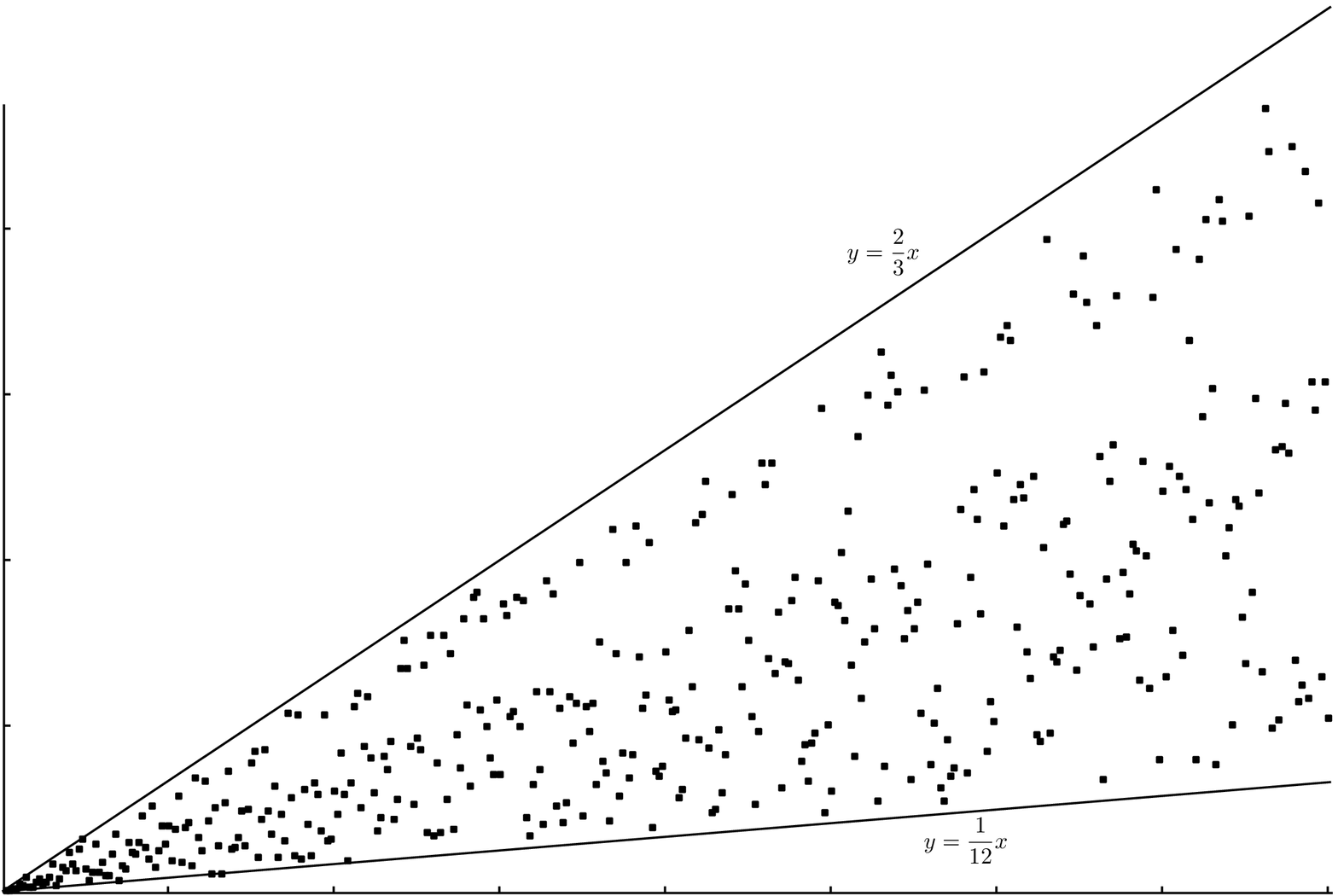}
		\caption{Construction $S_{\mathrm{mod2lex}<}$ for $n=800$}\label{fig:Smod2lex}
	\end{figure}
	
	\begin{conj} $S_{\mathrm{lex<}}$ and $S_{\mathrm{mod2lex<}}$ provide a point set such that the number of points in each square $[1,n]\times [1,n]$ can be bounded from below by a linear function.
	\end{conj}
	
	The consistent behavior of $S_{\mathrm{mod2lex<}}$ suggests that it might be more amenable to rigorous analysis than other greedy constructions.
	

	\paragraph{Acknowledgement.} 
	The paper resulted from the 11th Emléktábla Workshop, where Dömötör Pálvölgyi made us aware of the problem.   We are grateful to the organisers of the workshop.  We also thank Joshua Erde for first proposing the problem.
	
	\bibliography{main}

\providecommand{\bysame}{\leavevmode\hbox to3em{\hrulefill}\thinspace}
\providecommand{\MR}{\relax\ifhmode\unskip\space\fi MR }
\providecommand{\MRhref}[2]{%
  \href{http://www.ams.org/mathscinet-getitem?mr=#1}{#2}
}
\providecommand{\href}[2]{#2}
\begin{thebibliography}{10}

\bibitem{Adena/Holton/Kelly:1974}
Michael~A. Adena, Derek~A. Holton, and Patrick~A. Kelly, \emph{Some thoughts on
  the no-three-in-line problem}, Combinatorial mathematics ({P}roc. {S}econd
  {A}ustralian {C}onf., {U}niv. {M}elbourne, {M}elbourne, 1973), 1974,
  pp.~6--17. Lecture Notes in Math., Vol. 403.

\bibitem{Aichholzer/Eppstein/Hainzl:2023}
Oswin Aichholzer, David Eppstein, and Eva-Maria Hainzl, \emph{Geometric
  dominating sets - a minimum version of the {N}o-{T}hree-{I}n-{L}ine problem},
  Computational Geometry \textbf{108} (2023), Article 101913.

\bibitem{anderson1979update}
David~Brent Anderson, \emph{Update on the no-three-in-line problem}, Journal of
  Combinatorial Theory, Series A \textbf{27} (1979), no.~3, 365--366.

\bibitem{Baker/Harman/Pintz:2001}
R.~C. Baker, G.~Harman, and J.~Pintz, \emph{The difference between consecutive
  primes. {II}}, Proc. London Math. Soc. (3) \textbf{83} (2001), no.~3,
  532--562.

\bibitem{ball2005bounds}
Simeon Ball and James William~Peter Hirschfeld, \emph{Bounds on (n, r)-arcs and
  their application to linear codes}, Finite Fields and Their Applications
  \textbf{11} (2005), no.~3, 326--336.

\bibitem{ball2020arcs}
Simeon Ball and Michel Lavrauw, \emph{Arcs in finite projective spaces}, EMS
  Surveys in Mathematical Sciences \textbf{6} (2020), no.~1, 133--172.

\bibitem{bose1947mathematical}
Raj~Chandra Bose, \emph{Mathematical theory of the symmetrical factorial
  design}, Sankhy{\=a}: The Indian Journal of Statistics (1947), 107--166.

\bibitem{Brass2005}
Peter Brass, William O.~J. Moser, and János Pach, \emph{Lattice point
  problems}, pp.~417--433, Springer New York, New York, NY, 2005.

\bibitem{dudeney1958amusements}
Henry~E Dudeney, \emph{Amusements in mathematics}, vol. 473, Courier
  Corporation, 1917.

\bibitem{Eppstein:2018}
David Eppstein, \emph{Forbidden configurations in discrete geometry}, Cambridge
  University Press, 2018.

\bibitem{flammenkamp1998progress}
Achim Flammenkamp, \emph{Progress in the no-three-in-line problem, ii}, Journal
  of Combinatorial Theory, Series A \textbf{81} (1998), no.~1, 108--113.

\bibitem{Gardner:1976}
Martin Gardner, \emph{Mathematical games: combinatorial problems, some old,
  some new and all newly attacked by computer}, Scientific American
  \textbf{235} (1976), no.~4, 131--137.

\bibitem{Hall/Jackson/Sudbery/Wild:1975}
R.~R. Hall, T.~H. Jackson, A.~Sudbery, and K.~Wild, \emph{Some advances in the
  no-three-in-line problem}, J. Combinatorial Theory Ser. A \textbf{18} (1975),
  336--341. \MR{366817}

\bibitem{OEIS}
OEIS~Foundation Inc., \emph{The {O}n-line {E}ncyclopedia of {I}nteger
  {S}equences}, Published electronically at \url{https://oeis.org/}, 2022.

\bibitem{kim2003small}
J.H. Kim and V.H. Vu, \emph{Small complete arcs in projective planes},
  Combinatorica \textbf{2} (2003), no.~23, 311--363.

\bibitem{Payne/Wood:2013}
Michael~S. Payne and David~R. Wood, \emph{On the general position subset
  selection problem}, SIAM J. Discrete Math. \textbf{27} (2013), no.~4,
  1727--1733.

\bibitem{por2007no}
Attila P{\'o}r and David~R Wood, \emph{No-three-in-line-in-3{D}}, Algorithmica
  \textbf{47} (2007), no.~4, 481--488.

\bibitem{Roth:1951}
Karl~F. Roth, \emph{On a problem of {H}eilbronn}, J. London Math. Soc.
  \textbf{26} (1951), 198--204. \MR{41889}

\bibitem{Segre:1955a}
Beniamino Segre, \emph{Curve razionali normali e {$k$}-archi negli spazi
  finiti}, Ann. Mat. Pura Appl. (4) \textbf{39} (1955), 357--379.

\bibitem{Segre:1955}
\bysame, \emph{Ovals in a finite projective plane}, Canadian J. Math.
  \textbf{7} (1955), 414--416.

\bibitem{sloane2007line}
Neil~J.A. Sloane, \emph{The on-line encyclopedia of integer sequences}, Towards
  mechanized mathematical assistants, Springer, 2007, pp.~130--130.

\bibitem{Wood:2005}
David~R. Wood, \emph{Grid drawings of {$k$}-colourable graphs}, Comput. Geom.
  \textbf{30} (2005), no.~1, 25--28.

\end{thebibliography}
	\bibliographystyle{amsplain}

\end{document}